\documentclass[12pt]{amsart}
\usepackage{amsfonts,amsthm,amsmath,amssymb,verbatim}
\usepackage{graphicx}
\usepackage[active]{srcltx}
\usepackage[all,cmtip]{xy}

\newtheorem{thm}{Theorem}[section]

\newtheorem{lemma}[thm]{Lemma}

\theoremstyle{remark}

\newtheorem{example}[thm]{Example}
\newtheorem{remark}[thm]{Remark}



\def\C{\mathbb{C}}

\def\Z{\mathbb{Z}}
\def\P{\mathbb{P}}
\def\N{\mathbb{N}}

\def\R{\mathbb{R}}

\newcommand{\conv}{\mathop{\mathrm{conv}}}
\def\const{{\rm const}}

\def\Oc{\mathcal{O}}

\def\om{\omega}
\def\l{\lambda}
\def\s{\tau}

\def\wk{\underline{w_k}}
\def\w0{\underline{w_0}}
\def\id{\rm id}

\title{Newton--Okounkov polytopes of flag varieties}
\author{Valentina Kiritchenko}
\email{vkiritch@hse.ru}
\thanks{The research was carried out at the IITP RAS at the expense of the Russian
Foundation for Sciences (project 14-50-00150).}

\address{Laboratory of Algebraic Geometry and Faculty of Mathematics\\
National Research University Higher School of Economics\\
Vavilova St. 7, 112312 Moscow, Russia}
\address{Institute for Information Transmission Problems, Moscow, Russia}

\subjclass[2010]{14M15, 52B20, 05E10}

\date{}
\keywords{Newton--Okounkov polytope, flag variety, Bott--Samelson resolution, fiber polytope}

\begin{document}
\begin{abstract}
We compute the Newton--Okounkov bodies of line bundles on the complete flag variety of $GL_n$ for
a geometric valuation coming from a flag of translated Schubert subvarieties.
The Schubert subvarieties correspond to the terminal subwords in the decomposition
$(s_1)(s_2s_1)(s_3s_2s_1)(\ldots)(s_{n-1}\ldots s_1)$
of the longest element in the Weyl group.
The resulting Newton--Okounkov bodies coincide with the Feigin--Fourier--Littelmann--Vinberg
polytopes in type $A$.
\end{abstract}
\maketitle

\section{Introduction}
Newton--Okounkov convex bodies generalize Newton polytopes from toric geometry to a more general
algebro-geometric as well as representation-theoretic setting.
In particular, Newton--Okounkov bodies of flag varieties and of Bott--Samelson resolutions for
different valuations have recently attracted much interest due to connections with representation
theory and Schubert calculus.
The Newton--Okounkov body can be assigned to a line bundle on an algebraic variety $X$
\cite{KK,LM}.
In contrast with Newton polytopes, Newton--Okounkov bodies depend heavily on a choice of a
valuation on the field of rational functions $\C(X)$.
In the case of flag varieties, it is especially interesting to consider various geometric
valuations, namely, valuations coming from a complete flag of subvarieties
${pt}=Y_d\subset \ldots \subset Y_1\subset Y_0=X$, where $d:=\dim X$, since the resulting
Newton--Okounkov convex bodies can often be identified with polytopes that arise in representation
theory.

The first explicit computation of Newton--Okounkov polytopes of flag varieties is due
to Okounkov \cite{O}.
For a geometric valuation, he identified Newton--Okounkov polytopes of symplectic flag varieties
with symplectic Gelfand--Zetlin polytopes.
Since then several other computations were made for different valuations
\cite{An,Fu,FFL14,HY,Ka,K14}, see also \cite{An15,FK,SchS} for related results.
In the present paper, we use a natural geometric valuation introduced by Anderson in
\cite[Section 6.4]{An}
who computed an example for $GL_3$.
In this example, the Newton--Okounkov polytope was identified with the 3-dimensional Gelfand--Zetlin
polytope.

Let $X$ be the complete flag variety for $GL_n(\C)$.
We compute Newton--Okounkov convex bodies of semiample line bundles on $X$
for the geometric valuation coming from the flag of translated
Schubert subvarieties
$$w_0X_{\id}\subset w_0w_{d-1}^{-1}X_{w_{d-1}}\subset w_0w_{d-2}^{-1}X_{w_{d-2}}
\subset\ldots\subset w_0w_{1}^{-1}X_{w_{1}}\subset X,$$
where $w_1$, $w_2$,\ldots, $w_{d-1}$ are terminal
subwords of the decomposition
$$(s_1)(s_2s_1)(s_3s_2s_1)(\ldots)(s_{n-1}\ldots s_1)$$
of the longest element in $S_n$ (see Section \ref{ss.val} for a precise definition).
The valuation can be alternatively described as the lowest term valuation associated with a natural
coordinate system on the open Schubert cell in $X$ (see Section \ref{ss.coord}).
The computation is based on simple algebro-geometric and convex-geometric arguments.
The only representation-theoretic input is the well-known fact that the number of integer points
in the Gelfand--Zetlin polytope for a dominant weight $\lambda$ is equal to the dimension of the
irreducible representation of $GL_n$ with the highest weight $\l$.

Surprisingly, the resulting polytopes for $n>3$ are not, in general, combinatorially equivalent
to the Gelfand--Zetlin polytopes and coincide instead with Feigin--Fourier--Littelmann--Vinberg
polytopes in type $A$.
The complete list of cases when Feigin--Fourier--Littelmann--Vinberg
polytopes in type $A$ are combinatorially equivalent to the Gelfand--Zetlin polytopes can be found
in \cite{Fo}.
Though Feigin--Fourier--Littelmann--Vinberg polytopes can also be defined in type $C$
an analogous result for Newton--Okounkov polytopes does not hold already for $Sp_4(\C)$
(see Section \ref{ss.C} for more details).
In both types $A$ and $C$, Feigin--Fourier--Littelmann--Vinberg polytopes were earlier
obtained as Newton--Okounkov bodies for a completely different
valuation that does not come from any decomposition of the longest element
(see \cite[Examples 8.1,8.2]{FFL14}).
The fact that valuations considered in \cite{FFL14} and in the present paper
yield the same Newton--Okounkov polytopes served as the starting point for the recent preprint
\cite{FFL15}, which gives a conceptual explanation for this coincidence
(see \cite[Example 17]{FFL15}).

The paper is organized as follows.
In Section \ref{s.main},  we define the valuation, formulate the main result and consider several
examples.
Section \ref{s.proof} contains the proof of the main theorem modulo the result
on comparison between the Gelfand--Zetlin and Feigin--Fourier--Littelmann--Vinberg polytopes.
The latter result is explained in Section \ref{s.comparison} using purely convex-geometric
arguments.

I am grateful to Alexander Esterov, Evgeny Feigin and Evgeny Smirnov for useful discussions.
I would also like to thank the referee for valuable comments.

\section{Main result}\label{s.main}
In this section, we define the valuation on $\C(X)$, recall the inequalities
defining Feigin--Fourier--Littelmann--Vinberg polytopes and formulate the main
theorem.
We also define a geometrically natural coordinate system on the open Schubert cell and use it
do the simplest examples by hand.
Finally, we discuss the case of symplectic flag varieties.
\subsection{Valuation}\label{ss.val}
Fix the decomposition  $\w0=(s_1)(s_2s_1)(s_3s_2s_1)\ldots(s_{n-1}\ldots s_1)$ of the longest
element $w_0\in S_n$.
Here $s_i:=(i~i+1)$ is the $i$-th elementary transposition.
Denote by $d:=\binom{n}{2}$ the length of $w_0$.

Fix a complete flag of subspaces
$F^\bullet:=(F^1\subset F^2\subset\ldots\subset F^{n-1}\subset \C^n)$ (this amounts to
fixing a Borel subgroup $B\subset GL_n$).
In what follows, $\wk$ for $k=1$,\ldots, $d$ denotes the subword of $\w0$ obtained by deleting the first
$k$ simple reflections in $\w0$, and $w_k$ denotes the corresponding element of $S_n$.
Consider the flag of translated Schubert subvarieties:
$$w_0X_{\id}\subset w_0w_{d-1}^{-1}X_{w_{d-1}}\subset w_0w_{d-2}^{-1}X_{w_{d-2}}
\subset\ldots\subset w_0w_{1}^{-1}X_{w_{1}}\subset GL_n/B,\eqno (*)$$
where Schubert subvarieties are taken with respect to the flag $F^\bullet$, i.e.,
$X_w=\overline{BwB/B}$
(cf. \cite[Section 6.4]{An} and \cite[Remark 2.3]{Ka}).
Let $y_1$, \ldots, $y_d$ be coordinates on the open Schubert cell $C$ (with respect to $F^\bullet$)
that are compatible with $(*)$, i.e.,
$w_0w_{k}^{-1}X_{w_k}\cap C=\{y_1=\ldots=y_k=0\}$.
A possible choice of such coordinates is described in Section \ref{ss.coord}.

Fix the lexicographic ordering on monomials in coordinates $y_1$, \ldots, $y_d$, i.e.,
$y_1^{k_1}\cdots y^{k_d}\succ y^{l_1}\cdots y^{l_d}$ iff
there exists $j\le d$ such that $k_i=l_i$ for $i<j$ and $k_j>l_j$.
Let $v$ denote the lowest order term valuation on $\C(X_{\w0})=\C(GL_n/B)$
associated with these coordinates and ordering.
Let $L_\l$ be the line bundle on $GL_n/B$ corresponding to a dominant weight
$\l:=(\l_1,\ldots,\l_n)\in\Z^n$ of $GL_n$ ({\em dominant} means that $\l_1\ge\l_2\ge\ldots\ge \l_n$).
Recall that the bundle $L_\l$ is semiample iff $\l$ is dominant and very ample iff $\l$ is strictly dominant, i.e.,
$\l_1>\l_2>\ldots> \l_n$.
Denote by $\Delta_v(GL_n/B,L_\l)\subset\R^d$
the Newton--Okounkov convex body corresponding to $GL_n/B$,
$L_\l$ and $v$ (see \cite{KK,LM} for a definition of Newton--Okounkov convex bodies).

\begin{thm} \label{t.main}
The Newton--Okounkov convex body $\Delta_v(GL_n/B,L_\l)$ coincides with the
Feigin--Fourier--Littelmann--Vinberg polytope
$FFLV(\l)$.
\end{thm}
We now recall the definition of $FFLV(\l)$.
Label coordinates in $\R^d$ corresponding to
$(y_1,\ldots,y_d)$ by
$(u^1_{n-1};u^2_{n-2},u^1_{n-2};\ldots;u^{n-1}_1,u^{n-2}_{1},\ldots,u^{1}_1)$.
Arrange the coordinates into the table
$$
\begin{array}{cccccccccc}
\l_1&       & \l_2    &         &\l_3          &    &\ldots    & &       &\l_n   \\
    &u^1_1&         &u^1_2  &         & \ldots   &       &  &u^1_{n-1}&       \\
    &       &u^2_1 &       &  \ldots &   &        &u^2_{n-2}&         &       \\
    &       &       &  \ddots   & &  \ddots   &      &         &         &       \\
    &       &       &  &u^{n-2}_1&     &  u^{n-2}_2 &        &         &       \\
    &       &         &    &     &u^{n-1}_1&   &              &         &       \\
\end{array}
\eqno{(FFLV)}$$
The polytope $FFLV(\l)$ is defined by inequalities
$u^l_m\ge 0$ and
$$\sum_{(l,m)\in D}u^l_m\le \l_i-\l_j$$
for all Dyck paths going from $\l_i$ to $\l_j$ in table $(FFLV)$  where $1\le i<j\le n$
(see \cite{FFL} for more details).

\begin{example} \label{e.FFLV}
{\it (a)} For $n=3$, there are six inequalities
$$0\le u^1_1\le\l_1-\l_2; \quad 0\le u^1_2\le\l_2-\l_3; \quad 0\le u^2_1; \quad
u^1_1+u^2_1+u^1_2\le\l_1-\l_3.$$
In this case, there is a unimodular change of coordinates that maps $FFLV(\l)$ to
the Gelfand--Zetlin polytope $GZ(\l)$ (see Section \ref{s.comparison} for a definition of $GZ(\l)$).

{\it (b)} For $n=4$, there are $13$ inequalities
$$0\le u^1_1\le\l_1-\l_2; \quad 0\le u^1_2\le\l_2-\l_3; \quad 0\le u^1_3\le\l_3-\l_4;\quad
0\le u^2_1, \ u^2_2, \ u^3_1;$$
$$u^1_1+u^2_1+u^1_2\le\l_1-\l_3; \quad u^1_2+u^2_2+u^1_3\le\l_2-\l_4;$$
$$u^1_1+u^2_1+u^1_2+u^2_2+u^1_3\le\l_1-\l_4; \quad u^1_1+u^2_1+u^3_1+u^2_2+u^1_3\le\l_1-\l_4.$$
In this case, $FFLV(\l)$ and $GZ(\l)$ are combinatorially different whenever $\l$ is strictly
dominant because they have different number of facets (cf. \cite[Proposition 2.1.1]{Fo}).
\end{example}

\subsection{Coordinates}\label{ss.coord}
We now introduce coordinates on the open Schubert cell in $GL_n/B$ that are compatible with the
flag $(*)$.
These coordinates seem to be natural from a geometric viewpoint and will be used to compute by hand
some examples in the end of this section.
However, they are not needed for the proof of the main result.

To motivate the definition consider first the Bott--Samelson variety $X_{\w0}$.
Its points are collections of $d$ subspaces
$\{V^i_j\subset \C^n\ |\ i+j\le n, \ i,j>0 \}$ such that $\dim V^i_j=i$, and
$V^{i}_j$, $V^{i}_{j+1}\subset V^{i+1}_j$ where we put $V^{i+1}_{n-i}:=F^{i+1}$.
Incidence relations between subspaces $V^i_j$ can be organized into the following table
(similar to the Gelfand--Zetlin table).

$$
\begin{array}{cccccccccccc}
    &V^1_1  &     &V^1_2 &         &\ldots   &         &         &V^1_{n-1}&   &   &F^1\\
    &       &V^2_1&      &\ldots   &         &         &V^2_{n-2}&         &   &F^2&   \\
    &       &     &\ddots&         &\ddots   &         & \cdots  &         &   &   &   \\
    &       &     &      &V^{n-2}_1&         &V^{n-2}_2&         &F^{n-2}  &   &   &   \\
    &       &     &      &         &V^{n-1}_1&         &F^{n-1}      &         &   &   &   \\
\end{array}
$$
where the notation
$$
 \begin{array}{ccc}
  U &  &V \\
   & W &
 \end{array}
 $$
means $U, V\subset W$.

Collections of spaces $(V^i_j\subset \C^n\ |\ i+j\le n, \ i,j\ge1 )$ appear naturally
when we start from the fixed flag $F^\bullet$ and apply $d$ one parameter deformations to get
the moving flag $M^\bullet:=(V^1_1\subset V^2_1\subset\ldots\subset V^{n-1}_1\subset \C^n)$.
The deformations are encoded by the word $\w0$ as follows.
The elementary transposition $s_i$ corresponds to $\P^1$-family of complete flags
that differ only in the $i$-th subspace.
To go from $F^\bullet$ to $M^\bullet$ we first move $F^1$ inside $F^2$ and get the flag
$(V^{1}_{n-1}\subset F^2\subset\ldots\subset F^{n-1})$, second we move $F^2$ inside $F^3$
and get $(V^{1}_{n-1}\subset V^{2}_{n-2}\subset F^3\subset\ldots\subset F^{n-1})$, third we move
$V^{1}_{n-1}$ inside $V^{2}_{n-2}$ to get $V^1_{n-2}$ and so on.

\begin{example} \label{e.coord4}
Let $n=4$.
Below is the sequence of intermediate flags between $F^\bullet$ and $M^\bullet$.
$$F^{\bullet}\stackrel{s_1}{\to}(V^{1}_{3}\subset F^2\subset F^{3})
\stackrel{s_2}{\to}(V^{1}_{3}\subset V^2_2\subset F^{3})
\stackrel{s_1}{\to}(V^{1}_{2}\subset V^2_2\subset F^{3})
\stackrel{s_3}{\to}$$
$$(V^{1}_{2}\subset V^2_2\subset V^{3}_1)
\stackrel{s_2}{\to}(V^{1}_{2}\subset V^2_1\subset V^{3}_1)
\stackrel{s_1}{\to}M^{\bullet}$$
\end{example}

\begin{remark}
The word $\w0$ is the same (after switching $s_i$ and $s_{n-i}$) as the word used in \cite[2.2]{V} to encode the path
from the fixed flag to the moving flag in order to establish a geometric Littlewood--Richardson
rule for Grassmannians.
According to \cite[3.12]{V} not every reduced decomposition of $w_0$ can be used for this purpose
which is another manifestation of the special properties of $\w0$.
\end{remark}

Note that if $F^\bullet$ and $M^\bullet$ are in general position (that is, $M^\bullet$ lies in
the open Schubert cell $C$ with respect to $F^\bullet$), then all subspaces $V^i_j$ are uniquely
defined by $M^\bullet$, namely, $V^i_j=F^{n-j+1}\cap M^{i+j-1}$.
In particular, the natural projection
$$\pi_{\w0}:X_{\w0}\to GL_n/B; \quad \pi_{\w0}:(V^{i}_j)\mapsto M^\bullet$$
is one to one over $C$.
Fix a basis $e_1$,\ldots, $e_n$ in $\C^n$ compatible with $F^\bullet$, i.e.,
$F^i=\langle e_1,\ldots,e_i\rangle$ (fixing such a basis is equivalent to fixing a maximal torus
$T\subset B$, and hence, an action of the Weyl group on flags).
Using the word $\w0$ we now introduce natural coordinates
$(x^1_{n-1};x^2_{n-2},x^1_{n-2};\ldots;x^{n-1}_1,x^{n-2}_{1},\ldots,x^{1}_1)$
on $C\simeq \pi^{-1}_{\w0}(C)$.
The origin in this coordinate system is the flag $w_0F^\bullet:=(w_0F^1\subset w_0F^2\subset\ldots
\subset w_0F^{n-1})$.
The coordinate $x^i_j$ determines the position of $V^i_j$ inside the $\P^1$-family of dimension $i$
subspaces $V^i_j(x^i_j)$ such that $V^{i-1}_{j+1}\subset V^i_j(x^i_j)\subset V^{i+1}_{j}$.
To define the coordinate $x^i_j$ on $\P^1$
uniquely up to a constant factor it is enough to choose $V^i_j(0)$ and $V^i_j(\infty)$.
The following choice seems to be the most natural.

Since $M^\bullet$ and $F^\bullet$ are in general position, that is, $\dim (F^{n-j}\cap M^{i+j})=i$,
we have inclusions of pairwise distinct subspaces:
$$
\begin{array}{ccc}
   & V^{i-1}_{j+1}=F^{n-j}\cap M^{i+j-1}& \\
 V^i_j=F^{n-j+1}\cap M^{i+j-1} & \ne &V^i_{j+1}=F^{n-j}\cap M^{i+j} \\
   & V^{i+1}_{j}=F^{n-j+1}\cap M^{i+j} &\\
\end{array}
$$
Put $V^i_j(\infty):=V^i_{j+1}$ and
$V^i_j(0):=\langle F^{n-i-j}, e_{n-j+1} \rangle\cap M^{i+j}+V^{i-1}_{j+1}$.
Note that $\langle F^{n-i-j}, e_{n-j+1} \rangle\cap M^{i+j}$ is the line spanned by a vector
$e_{n-j+1}+v$ for some $v\in F^{n-i-j}$ since $F^{n-i-j}\cap M^{i+j}=\{0\}$.
It follows that $\dim V^i_j(0)=i$, and $V^i_j(0)\ne V^i_j(\infty)$ because $e_{n-j+1}\notin F^{n-j}$.
By construction, $V^{i-1}_{j+1}\subset V^i_j(0)\subset V^{i+1}_{j}$.
Note also that $V^i_j$ lies in $\mathbb A^1=\P^1\setminus\{V^i_j(\infty)\}$ when
$M^\bullet$ and $F^\bullet$ are in general position.

\begin{remark}\label{r.flag}
It is not hard to check that coordinates
$(y_1,\ldots,y_d):=(x^1_{n-1};x^2_{n-2},x^1_{n-2};\ldots;x^{n-1}_1,x^{n-2}_{1},\ldots,x^{1}_1)$
are compatible with the flag $(*)$ of Schubert subvarieties.
\end{remark}

\begin{example}\label{e.coord3} Let $n=3$.
Then
$$V^1_1=\langle (x^1_1x^1_2-x^2_1)e_1+x^1_1e_2+e_3\rangle; \quad
V^1_2=\langle x^1_2e_1+e_2\rangle;$$
$$V^2_1=\langle x^1_2e_1+e_2, -x^2_1e_1+e_3\rangle.$$
Figure 1 depicts projectivizations in $\P^2$ of various subspaces involved in this example.
\begin{figure}
\includegraphics[width=10cm]{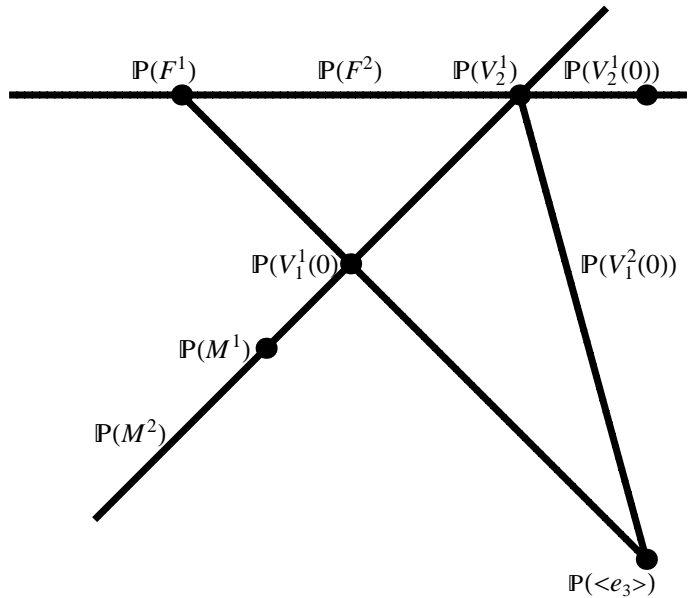}
\caption{Coordinates on flags for $n=3$.}
\end{figure}
\end{example}

\subsection{Examples}
Theorem \ref{t.main} will be proved in the next section.
Here we verify it by hand in three simplest examples.
\begin{example}{cf. \cite[Section 6.4]{An}}
Let $n=3$, and $\l=(2,1,0)$.
The flag variety $GL_3/B$ can be regarded as a hypersurface in $\P^2\times {\P^2}^*$ under
the embedding $(V^1_1, V^2_1)\mapsto V^1_1\times V^2_1$.
The line bundle $L_\l$ on $GL_3/B$ is the pullback of the dual tautological line bundle $\Oc(1)$
on $\P^8$ under the embedding:
$$p_\l:GL_3/B\hookrightarrow\P^2\times {\P^2}^*\stackrel{\mbox{\tiny Segre}}{\longrightarrow}\P^{8}.
$$
Using Example \ref{e.coord3} we get that in coordinates
$(y_1,y_2,y_3)=(x^1_2,x^2_1,x^1_1)$ the map $p_\l$ takes the form
$$p_\l:(y_1,y_2,y_3)\mapsto \begin{pmatrix}
y_1y_3-y_2\\
y_3\\
1\\
\end{pmatrix}
\times
\begin{pmatrix}
y_2&&y_1&&1\\
\end{pmatrix}.$$
Hence, $H^0(GL_3/B,L_\l)$ has the basis $1$, $y_1$, $y_2$, $y_3$, $y_1y_3$, $y_2y_3$,
$y_1y_2y_3-y_2^2$, $y_1^2y_3-y_1y_2$.
Applying the valuation $v$ we get $8$ integer points $(0,0,0)$, $(1,0,0)$, $(0,1,0)$, $(0,0,1)$,
$(1,0,1)$, $(0,1,1)$, $(0,2,0)$, $(1,1,0)$, whose convex hull in $\R^3$ is given exactly by
inequalities of Example \ref{e.FFLV}{\it (a)}.
\end{example}
\begin{example} Let $n=4$, and $\l=(1,1,0,0)$.
The line bundle $L_\l$ on $GL_4/B$ is the pullback of the dual tautological line bundle $\Oc(1)$
on $\P^5$ under the natural projection $GL_4/B\to G(2,4)$ composed with the Pl\"ucker
embedding $G(2,4)\hookrightarrow\P^5$ of the Grassmannian.
Using Example \ref{e.coord4} we get that in coordinates $(y_1,\ldots,y_6)$
the plane $V^2_1$ is spanned by the vectors $(y_4y_6+y_5,y_4,1,0)$ and $(y_2y_6+y_3,y_2,0,1)$.
Hence, the map $p_\l$ has the form
$$p_\l:(y_1,\ldots,y_6)\mapsto (y_2y_5-y_3y_4: -(y_2y_6+y_3): y_4y_6+y_5: -y_2: y_4: 1).$$
The valuation $v$ takes the sections of $H^0(GL_4/B,L_\l)$ to 6 integer points
in the 4-space $\{u^1_1=u^1_3=0\}$.
In coordinates $(u^2_1,u^3_1,u^1_2,u^2_2)$, these points are $(0,1,1,0)$, $(0,1,0,0)$, $(0,0,0,1)$,
$(1,0,0,0)$, $(0,0,1,0)$, $(0,0,0,0)$.
Their convex hull in $\R^4$ is given exactly by inequalities of Example \ref{e.FFLV}{\it (b)}.
\end{example}

\begin{example} \label{e.Grass}
The previous example can be extended to $G(3,6)$, that is, $n=6$ and
$\l=(1,1,1,0,0,0)$.
This is the minimal example when $FFLV(\l)$ and $GZ(\l)$ are
not combinatorially equivalent (cf. \cite[Proposition 2.1.1]{Fo}).
When computing $V^3_1$ in coordinates $(y_1,\ldots,y_{15})$ one can immediately ignore all
monomials that contain $y_{15}$, $y_{14}$, $y_{13}$ since they never appear as the lowest
order terms.
The same holds for $y_3$, $y_2$, $y_1$.
If $y_{15}=y_{14}=y_{13}=0$, then $p_\l$ takes the following simple form:
$$p_\l:(y_4,\ldots,y_{12})\mapsto 3\times3 \mbox{ minors of }
\begin{pmatrix}
y_{10}&y_{11}&y_{12}&1&0&0\\
y_7&y_8&y_9&0&1&0\\
y_4&y_5&y_6&0&0&1\\
\end{pmatrix}.$$
Hence, we have to compute the lowest order terms of all minors of the $3\times 3$ matrix formed
by the first three columns.
After rotating this matrix as follows
$$
\begin{matrix}&&y_{10}&&&&&\\
&y_7&&y_{11}&&&&\\
y_4&&y_8&&y_{12}&&&\\
&y_5&&y_9&& &&\\
&&y_6&&&&&\\
\end{matrix}$$
it is easy to see that the lowest order monomials in the minors are in bijective correspondence with
those collections of $u^i_j$ (where $3\le i+j\le 6$, $j\le3$) in table $(FFLV)$ that can not occur
in the same Dyck path.
By definition, $FFLV(\l)$ contains an integer point with $u^i_j=1$ and $u^l_m=1$ iff no Dyck path
passes through both $u^i_j$ and $u^l_m$.
Hence, the valuation $v$ maps bijectively the minors to the integer points in $FFLV(\l)$.
\end{example}
\begin{remark} Arguments of Example \ref{e.Grass} allow one to identify
$\Delta_v(GL_n/B,L_{\om_i})$ with $FFLV(\om_i)$ for any fundamental weight $\om_i$ of $GL_n$.
This might lead to an alternative proof of Theorem \ref{t.main} if one uses that
$\Delta_v(GL_n/B,L_\l)$ for $\l=k_1\om_1+\ldots+k_{n-1}\om_{n-1}$
contains the Minkowski sum $k_1\Delta_v(GL_n/B,L_{\om_1})+\ldots+
k_{n-1}\Delta_v(GL_n/B,L_{\om_{n-1}})$.
\end{remark}

\subsection{Symplectic case}\label{ss.C}
A statement analogous to Theorem \ref{t.main} does not hold in type $C$
already in the case of $Sp_4$.
We now discuss this case in more detail.
For the rest of this section, $X$ denotes the complete flag variety for $Sp_4$.
The flag of translated Schubert subvarieties analogous to $(*)$ has the form
$$s_1s_2s_1s_2X_{\id}\subset s_1s_2s_1X_{s_2}\subset s_1s_2X_{s_1s_2}\subset
s_1X_{s_2s_1s_2}\subset X,$$
where $s_1$, $s_2$ are simple reflections.
The resulting Newton--Okounkov polytopes were computed in \cite[Proposition 4.1]{K14}.
Regardless of whether $s_1$ corresponds to the shorter or the longer root,
these polytopes have 11 vertices (for a strictly dominant weight)
while Feigin--Fourier--Littelmann--Vinberg polytopes (as well as string polytopes) for $Sp_4$
have 12 vertices.
In particular, the former are not combinatorially equivalent to the latter.

Note that the string polytopes for the decomposition
$$\w0=(s_1)(s_2s_1s_2)(\ldots)(s_n s_{n-1}\ldots s_2 s_1 s_2\ldots s_{n-1}s_n), \eqno(Sp)$$
where $s_1$ corresponds to the longer root,
coincides (after a unimodular change of coordinates) with the symplectic
Gelfand--Zetlin polytopes by \cite[Corollary 6.3]{L}.
The latter were exhibited in \cite{O} as the Newton--Okounkov bodies of the symplectic flag
variety $Sp_{2n}/B$
for the lowest term valuation associated with the $B$-invariant flag of
(not translated) Schubert subvarieties corresponding to the initial subwords of $\w0$:
$$X_{\id}\subset X_{w_0w_{1}^{-1}}\subset\ldots\subset X_{w_0w_{d-1}^{-1}}\subset Sp_{2n}/B,$$
where $d=n^2=\dim Sp_{2n}/B$.

Finally, note that string polytopes for any connected reductive group $G$ and any reduced
decomposition $\w0$ were obtained in \cite{Ka} as the Newton--Okounkov bodies of the complete
flag variety $G/B$ for the {\bf highest} term valuation associated with
the $B$-invariant flag of Schubert subvarieties:
$$X_{\id}\subset X_{w_{d-1}}\subset\ldots\subset X_{w_1}\subset G/B.$$
Here $d$ denotes the dimension of $G/B$ (and the length of $\w0$).
Note that for $G=GL_n$ and $\w0$ as in Section \ref{ss.val},
the string polytope coincides with the Gelfand--Zetlin polytope in type $A$ by \cite[Corollary 5.2]{L}.
 While the highest term valuation comes naturally when dealing with crystal bases and string
polytopes the lowest term valuation is more natural from a geometric viewpoint since it can be
interpreted using the order of the pole of a rational function along a hypersurface.
\section{Proof of Theorem \ref{t.main}}
\label{s.proof}
We first formulate and prove simple general results about Newton--Okounkov bodies and recall
classical facts about divisors on Schubert varieties.
Then we prove Theorem \ref{t.main}.
\subsection{Preliminaries}
We will need the following two simple lemmas on Newton--Okounkov convex bodies.
\begin{lemma}\label{l.effective} Let $X$ be a variety, $L$ a line bundle on $X$, and
$v$ a valuation on $\C(X)$.
If $D$ is an effective divisor on $X$, then
$$\Delta_v(X,L)\subset\Delta_v(X,L\otimes \Oc(D)).$$
\end{lemma}
\begin{proof} Since $D$ is effective, $1\in H^0(X,\Oc(D))$.
The lemma follows directly from the definition of Newton--Okounkov bodies since
for any $l\in\N$ we have the inclusion $i:H^0(X,L^{\otimes l})\subset H^0(X,(L\otimes \Oc(D))^{\otimes l})$ given by
$i(s)=s\otimes 1$.
\end{proof}
The lemma below is a partial case of \cite[Theorem 4.24]{LM}.
We provide a short proof for the reader's convenience.
\begin{lemma}\label{l.projective}
Let $X\subset\P^N$ be a projective variety of dimension $d$, and $Y_\bullet=
(\{x_0\}=Y_d\subset \ldots \subset Y_1\subset Y_0=X)$ a complete flag of subvarieties
at a smooth point $x_0\in X$.
Consider a valuation $v$ on $\C(X)$ associated with the flag $Y_\bullet$, and the corresponding
coordinates $a_1,\ldots,a_d$ on $\R^d$.
Let $v_1$ be the restriction of the valuation $v$ to $\C(Y_1)$.
Denote by $L$ the restriction of the dual tautological bundle $\Oc_{\P^N}(1)$ to $X$.
Then we have
$$\Delta_{v_1}(Y_1,L|_{Y_1})=\Delta_{v}(X,L)\cap\{a_1=0\}.$$
\end{lemma}
\begin{proof} It is well-known that the natural restriction map
$H^0(\P^N,\Oc_{\P^N}(l))\to H^0(X,L^{\otimes l})$ is surjective for sufficiently large $l$.
Similarly, the map $H^0(\P^N,\Oc_{\P^N}(l))\to H^0(Y_1,L^{\otimes l}|_{Y_1})$ is surjective.
Hence, the map $H^0(X,L^{\otimes l})\to H^0(Y_1,L^{\otimes l}|_{Y_1})$ is surjective, and
$\Delta_{v_1}(Y_1,L|_{Y_1})\subset\Delta_{v}(X,L)$.
For a section $s\in H^0(X,L^{\otimes l})$, denote by $\bar s$ its restriction to $Y_1$.
Then $\bar s\ne 0$ iff $v(s)\in\{a_1=0\}$.
Hence, $\Delta_{v_1}(Y_1,L|_{Y_1})=\Delta_{v}(X,L)\cap\{a_1=0\}$ as desired.
\end{proof}
We will also use the classical Chevalley formula \cite[Proposition 1.4.3]{B} and the description
of Cartier divisors on Schubert varieties \cite[Proposition 2.2.8]{B}.
When applied to $X_{w}$ from $(*)$ and $L_\l$ these propositions immediately yield the following
\begin{lemma}\label{l.Chevalley}
Let $w=(s_i\ldots s_1)(s_{n-j+1}\ldots s_1)\ldots(s_{n-1}\ldots s_1)$ where
$i+j\le n$.
Then the Picard group of $X_w$ is spanned by the classes of $X_{ws}$ where
$s$ runs through transpositions $s_1$, $s_2$\ldots,
$s_{j-1}$; $(j~j+1)$, $(j~j+2)$,\ldots, $(j~i+j)$ and
$(j-1~i+j+1)$, $(j-1~i+j+2)$,\ldots, $(j-1~ n)$.
In particular,
$$L_\l|_{X_{w}}=\bigotimes_{l=1}^{j-1}\Oc(X_{ws_l})^{\l_l-\l_{l+1}}\otimes
\bigotimes_{l=1}^{i}\Oc(X_{w(j~l+j)})^{\l_j-\l_{l+j}}\otimes$$
$$\otimes\bigotimes_{l=i+j+1}^{n}\Oc(X_{w(j-1~l)})^{\l_{j-1}-\l_{l}}.$$
\end{lemma}
\begin{remark} Lemma \ref{l.Chevalley} implies the following important property of
the decomposition  $\w0$.
For every $k\le d$, the Schubert subvariety $X_{w_k}$ is a Cartier divisor on $X_{w_{k-1}}$.
This property is used in the proof below.
It would be interesting to find decompositions with this property for other reductive groups
(decomposition $(Sp)$ for $Sp_n$ does not have this property).

Moreover, it is easy to check that all $X_{w_k}$ are smooth by \cite[Theorem 3.7.5]{M} but this is
not used in the proof.
\end{remark}
\subsection{Proof of Theorem \ref{t.main}} We will prove by induction the following more general statement.
Put $Y_k:=w_0w_{k}^{-1}X_{w_k}$, and let $v_k$ be the restriction of the valuation $v$ to
$\C(Y_k)\simeq\C(y_{k+1},\ldots,y_{d})$ (see Remark \ref{r.flag}).
We will also use an alternative labeling of coordinates in $\R^d$, namely, $(a_1,a_2,\ldots,a_d)=
(u^1_{n-1};u^2_{n-2},u^1_{n-2};\ldots;u^{n-1}_1,u^{n-2}_{1},\ldots,u^{1}_1)$.
Let $F_k(\l)$ be the face of $FFLV(\l)$ given by equations $u^l_m=0$ for all pairs $(l,m)$ such
that either $m>j$, or $m=j$ and $l\ge i$.
Here $k$ and $(i,j)$ are related  via the above identification of coordinates $a_k$
and $u^i_j$, i.e., $a_k=u^i_j$.

\begin{thm}\label{t.gmain} The Newton--Okounkov convex body $\Delta_{v_k}(Y_k,L_\l|_{Y_k})$ coincides
with the face $F_k(\l)$.
\end{thm}
In particular, this theorem reduces to Theorem \ref{t.main} when $k=0$ (we put $F_0(\l)=FFLV(\l)$).
The main idea of the proof is to identify the slices of $\Delta_{v_{k-1}}(Y_{k-1},L_\l|_{Y_{k-1}})$
by hyperplanes $\{a_k=\const\}$ with $F_k(\mu)$ for suitable $\mu$.
We will need a convex-geometric lemma for slices of $F_{k-1}(\l)$ and a similar
algebro-geometric lemma for $\Delta_{v_{k-1}}(Y_{k-1},L_\l|_{Y_{k-1}})$.

\begin{lemma}\label{l.CG}
There exists a path of dominant weights $\mu(t)$
such that
$$
(t-\l_{i+j}) e_{k}+F_k(\mu(t))=
F_{k-1}(\l)\cap\{a_k=t-\l_{i+j}\}.
$$
for all $t\in[\l_{i+j},\l_j]$.
Here $e_k$ denotes the $k$-th basis vector in $\R^d$.
In particular,
$$F_{k-1}(\l)=\conv\{(t-\l_{i+j}) e_{k}+F_k(\mu(t)) \ |\ \l_{i+j}\le t \le \l_j\}.$$
\end{lemma}
\begin{proof}
Define $\mu(t)=(\mu_1(t),\ldots,\mu_n(t))$ as follows
$$\mu_l(t)=\left\{
\begin{array}{ll}
\max\{\lambda_l,t\}&
\mbox{ if } j< l\le i+j\\
\l_l&
\mbox{ otherwise }\\
\end{array}
\right.$$
In particular, $\l=\mu(\l_{i+j})$, and every $\mu_l(t)$ is a piecewise linear concave function
of $t$.
The lemma now follows immediately from the definitions of $F_k(\l)$ and $FFLV(\l)$.
\end{proof}
In particular, $F_{k-1}(\l)$ fibers over the segment $[0,\l_j-\l_{i+j}]$, and the
fiber polytope is analogous to $F_k(\l)$ for strictly dominant $\lambda$.
\begin{lemma}\label{l.AG}
Take $\mu(t)$ as in the proof of Lemma \ref{l.CG}.
Then
$$(t-\l_{i+j}) e_{k}+\Delta_{v_{k}}(Y_{k},L_{\mu(t)}|_{Y_{k}})\subset
\Delta_{v_{k-1}}(Y_{k-1},L_\l|_{Y_{k-1}})\cap\{a_k=t-\l_{i+j}\}
$$
for all integer $t\in[\l_{i+j},\l_j]$.
In particular,
$$
\conv\{(t-\l_{i+j}) e_{k}+\Delta_{v_{k}}(Y_{k},L_{\mu(t)}|_{Y_{k}}) \ |\ \l_{i+j}\le t \le \l_j,
\ t\in\Z \}
\subset\Delta_{v_{k-1}}(Y_{k-1},L_\l|_{Y_{k-1}}).
$$
\end{lemma}
\begin{proof}
By definition, $Y_k$ and $Y_{k-1}$ are translates of the Schubert varieties $X_{w_k}$ and
$X_{w_{k-1}}$,
respectively, where $w_k=(s_{i-1}\ldots s_1)(s_{n-j+1}\ldots s_1)\ldots(s_{n-1}\ldots s_1)$ and
$w_{k-1}=s_iw_k$.
Put $\s=t-\l_{i+j}$.
It is easy to check using Lemma \ref{l.Chevalley} that
$$L_\l|_{Y_{k-1}}\otimes\Oc(-\s Y_k)=L_{\mu(t)}|_{Y_{k-1}}\otimes\Oc(\s(s_i Y_k -Y_k))\otimes E(\s)$$
for an effective Cartier divisor $E(\s)$ on $Y_{k-1}$.
Indeed,
$E(\s)=L_{(\l-\mu(t))}|_{Y_{k-1}}\otimes\Oc(-\s s_i Y_k )$
is a translate of the following divisor on $X_{w_{k-1}}$:
$$\bigotimes_{l=1}^{i-1}\Oc(X_{w(j~l+j)})^{\max\{0,t-\l_{l+j}\}}.$$
Note that $\Delta_{v_{k-1}}(Y_{k-1},L_{\mu(t)}|_{Y_{k-1}}\otimes\Oc(\s(s_i Y_k -Y_k)))=
\s e_{k}+\Delta_{v_{k-1}}(Y_{k-1}, L_{\mu(t)}|_{Y_{k-1}})$ since $s_i Y_k -Y_k$ is the divisor
of the rational function $y_k$.
Applying Lemma \ref{l.effective} to $Y_{k-1}$, $L_{\mu(t)}|_{Y_{k-1}}\otimes\Oc(\s(s_i Y_k -Y_k))$
and $E(\s)$ we get
$$\s e_{k}+\Delta_{v_{k-1}}(Y_{k-1}, L_{\mu(t)}|_{Y_{k-1}})\subset\Delta_{v_{k-1}}(Y_{k-1},
L_\l|_{Y_{k-1}}\otimes\Oc(-\s Y_k)).$$
Intersecting both sides with the hyperplane $\{a_k=\s\}$ yields
$$\s e_k+\Delta_{v_{k-1}}(Y_{k-1}, L_{\mu(t)}|_{Y_{k-1}})\cap\{a_k=0\}\subset\Delta_{v_{k-1}}(Y_{k-1},
L_\l|_{Y_{k-1}}\otimes\Oc(-\s Y_k))\cap\{a_k=\s\}.$$
Since $L_{\mu(t)}$ is semiample we can apply Lemma \ref{l.projective}
and get that
$$\Delta_{v_{k}}(Y_k,L_{\mu(t)}|_{Y_{k}}) =\Delta_{v_{k-1}}(Y_{k-1}, L_{\mu(t)}|_{Y_{k-1}})\cap \{a_k=0\}.$$
It follows that
$$\s e_k+ \Delta_{v_{k}}(Y_{k}, L_{\mu(t)}|_{Y_{k}})\subset\Delta_{v_{k-1}}(Y_{k-1},
L_\l|_{Y_{k-1}}\otimes\Oc(-\s Y_k))\cap \{a_k=\s\}.$$
It remains to note that $\Delta_{v_{k-1}}(Y_{k-1},
L_\l|_{Y_{k-1}}\otimes\Oc(-\s Y_k))\subset \Delta_{v_{k-1}}(Y_{k-1},
L_\l|_{Y_{k-1}})$ by Lemma \ref{l.effective}.
\end{proof}
We are now ready to prove Theorem \ref{t.gmain}.
\begin{proof}[Proof of Theorem \ref{t.gmain}] Let us first prove that
$F_{k}(\l)\subset\Delta_{v_{k}}(Y_{k},L_\l|_{Y_{k}})$ for all dominant $\l$
by backward induction on $k$.
For $k=d$, we have that both convex bodies coincide with the origin in $\R^d$.
Suppose the inclusion holds for $k$.
We now prove it for $k-1$.
By Lemma \ref{l.CG}
$$F_{k-1}(\l)=\conv\{(t-\l_{i+j}) e_{k}+F_k(\mu(t)) \ |\ \l_{i+j}\le t \le \l_j\}.$$
Moreover, when taking the convex hull it is enough to consider only integer values of $t$,
since $\mu(t)$ is linear at all non-integer points.
Using the induction hypothesis $F_k(\mu(t))\subset\Delta_{v_k}(Y_k,L_{\mu(t)}|_{Y_k})$
we get that
$$F_{k-1}(\l)\subset\conv\{(t-\l_{i+j}) e_{k}+\Delta_{v_{k}}(Y_{k},L_{\mu(t)}|_{Y_{k}}) \ |\
\l_{i+j}\le t \le \l_j,\ t\in\Z \}.$$
Hence, $F_{k-1}(\l)\subset\Delta_{v_{k-1}}(Y_{k-1},L_\l|_{Y_{k-1}})$ by Lemma \ref{l.AG}.

Finally, for $k=0$ we get $F_0(\l)\subset\Delta_{v}(GL_n/B,L_\l)$.
Since both convex bodies have the same volume they must coincide.
Here we use that by Theorem \ref{t.comparison} the volume of $F_0(\l)=FFLV(\l)$
coincides with the volume of the Gelfand--Zetlin polytope $GZ(\l)$.
Hence, inclusions $F_{k}(\l)\subset\Delta_{v_{k}}(Y_{k},L_\l|_{Y_{k}})$ are
equalities for all $k$.
\end{proof}
\begin{remark} Results of Section \ref{s.comparison} (see Theorem \ref{t.comparison} and
Remark \ref{r.flag_GZ}) imply that
the number of integer points in $F_{k}(\l)$ (and hence, in the Newton--Okounkov polytope
$\Delta_{v_{k}}(Y_{k},L_\l|_{Y_{k}})$) is equal to the dimension of the Demazure module
$H^0(Y_{k},L_\l|_{Y_{k}})$ for all $k=0,\ldots,d$ and dominant $\l$.
\end{remark}
To illustrate the proof of Theorem \ref{t.gmain} consider the simplest meaningful example.
\begin{example} Let $k=d-1$, i.e., $w_k=s_1$ and $w_{k-1}=s_2s_1$.
Then $Y_{k-1}=\hat\P^2$ is the blow up of $\P^2$ at one point, and $Y_{k}=\P^1$ is embedded into
$Y_{k-1}$ as one of the fibers of the $\P^1$-bundle $\hat\P^2\to\P^1$.
The Picard group of $\hat\P^2$ is spanned by $\Oc(Y_k)$ and $\Oc(E)$ where $E\subset\hat\P^2$
is the exceptional divisor.
Note that $\Oc(E)^a\otimes\Oc(Y_k)^b$ is semiample iff $0\le a\le b$.
We have
$$L_\l|_{Y_{k-1}}=\Oc(E)^{\l_1-\l_2}\otimes\Oc(Y_k)^{\l_1-\l_3}.$$
Hence, the line bundle $L_\l|_{Y_{k-1}}\otimes\Oc(-(t-\l_3) Y_k))$ is no longer semiample if
$\l_2<t\le\l_1$.
However, it has the same global sections (modulo multiplication by $y_{k}^{t-\l_3}$) as the
semiample bundle $L_{\mu(t)}=\Oc(E)^{\l_1-t}\otimes\Oc(Y_k)^{\l_1-t}$.
Hence, $L_{\mu(t)}$ can be used instead of $L_\l|_{Y_{k-1}}\otimes\Oc(-(t-\l_3) Y_k))$ when
computing $\Delta_{v_{k-1}}(L_\l|_{Y_{k-1}},Y_{k-1})$.
Figure 2 shows the Newton--Okounkov polygons of $L_\l|_{Y_{k-1}}$ (trapezoid) and
$L_{\mu(t)}|_{Y_{k-1}}$ (triangle), which are just Newton polygons since $Y_{k-1}$ is toric.
\end{example}
\begin{figure}
\includegraphics[width=7cm]{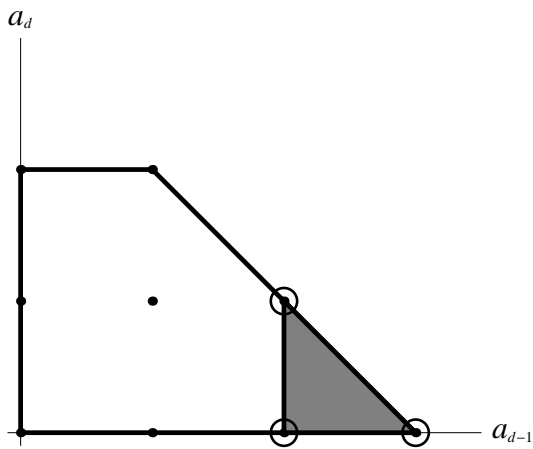}
\caption{Newton polygons of $L_\l|_{Y_{d-2}}$ and $L_{\mu(t)}|_{Y_{d-2}}$ for $d=3$, $\l=(3,1,0)$ and $t=2$.}
\end{figure}

\section{Comparison of Gelfand--Zetlin polytopes and Feigin--Fourier--Littelmann--Vinberg polytopes}
\label{s.comparison}
We start with an elementary construction of polytopes fibered over a segment.
Then we apply this construction to get the Gelfand--Zetlin and Feigin--Fourier--Littelmann--Vinberg
polytopes in a uniform way.
\subsection{Construction with fiber polytope}\label{ss.fiber}
Let $P\subset\R^l$ be a convex polytope.
The set of linear functionals, whose restrictions to $P$ attain their maximal values at a face
$F\subset P$, form a cone $C_F$; the {\em normal fan} of $P$ is defined as the
set of cones $C_F$ corresponding to all faces $F\subseteq Q$.
We say that a polytope $Q\subset\R^l$ is {\em subordinate} to $P$ if the normal fan of $P$ is
a subdivision of the normal fan of $Q$.
Note that the set of all polytopes subordinate to $P$ forms a semigroup under the Minkowski sum.
Denote this semigroup by $S_P$.

Let $\mu(t)$ be a piecewise-linear continuous function from a segment $I\subset \R$ to $S_P$.
We say that $\mu(t)$ is {\em convex} if
$$\frac{\mu(t_1)+\mu(t_2)}{2}\subset \mu\left(\frac{t_1+t_2}{2}\right)$$
for all $t_1,t_2\in I$.
In other words, the set
$$P_\mu:=\bigcup_{t\in I}\mu(t)\times \{t\}\subset \R^l\times \R=\R^{l+1}$$
is a convex polytope.
In this case, $P_\mu$ fibers over $I$ and the fiber polytope is subordinate to $P$.

Suppose now that $\mu'(t)$ is a convex function from $I$ to $S_Q$ for a convex polytope
$Q\subset\R^l$.
If the polytopes $\mu(t)$ and $\mu'(t)$ have the same Ehrhart polynomials for all $t\in I$
then obviously so do $P_\mu$ and $P_{\mu'}$.
The simplest example is when $P=Q$ and $\mu'(t)$ is a parallel translate of $\mu(t)$.
In this case, $P_\mu$ and $P_{\mu'}$ also have the same fiber polytope but might be combinatorially different
even for quite simple $\mu(t)$ and $\mu'(t)$ (see Example \ref{e.FvsG}).
\subsection{$GZ(\l)$ vs $FFLV(\l)$}
We now show that both $GZ(\l)$ and $FFLV(\l)$ can be obtained inductively from a point
using the above construction.
Recall that the Gelfand--Zetlin polytope $GZ(\l)\subset\R^d$ is defined by the following
inequalities
$$
\begin{array}{cccccccccc}
\l_1&       & \l_2    &         &\l_3          &    &\ldots    & &       &\l_n   \\
    &z^1_1&         &z^1_2  &         & \ldots   &       &  &z^1_{n-1}&       \\
    &       &z^2_1 &       &  \ldots &   &        &z^2_{n-2}&         &       \\
    &       &       &  \ddots   & &  \ddots   &      &         &         &       \\
    &       &       &  &z^{n-2}_1&     &  z^{n-2}_2 &        &         &       \\
    &       &         &    &     &z^{n-1}_1&   &              &         &       \\
\end{array}
$$
where the notation
$$
 \begin{array}{ccc}
  a &  &b \\
   & c &
 \end{array}
 $$
means $a\ge c\ge b$.
Let $G_k(\l)$ be the face of the Gelfand--Zetlin polytope $GZ(\l)$ given by
the equations $z^l_m=z^{l-1}_{m+1}$ for all pairs $(l,m)$ such that either $m>j$, or $m=j$ and
$l\ge i$ (we put $z^0_m=\l_m$).

\begin{remark}\label{r.flag_GZ}
In \cite[Theorem 3.4]{K}, there is an inductive construction of the Gelfand--Zetlin polytope
via convex geometric Demazure operators.
The flag of faces
$$G_d(\l)\subset G_{d-1}(\l)\subset G_{d-2}(\l)\subset\ldots\subset G_1(\l)\subset GZ(\l)=:G_0(\l).$$
is exactly the flag used in this construction.
In particular, by \cite[Corollary 4.5]{K} the number of integer points in $G_k$ is equal
to the dimension of the Demazure module $H^0(Y_{k},L_\l|_{Y_{k}})$ for all $k=0,\ldots,d$ and
dominant $\l$.
\end{remark}

\begin{lemma}\label{l.CG_GZ}
Take $\mu(t)$ as in the proof of Lemma \ref{l.AG}.
There exists a path $z(t)\in\R^d$
such that
$$G_{k-1}(\l)\cap\{z^i_j=t\}=z(t)+G_k(\mu(t))$$
for all integer $t\in[\l_{i+j},\l_j]$.
In particular,
$$G_{k-1}(\l)=\conv\{z(t)+G_k(\mu(t))\ |\ \l_{i+j}\le t \le \l_j\}.$$
\end{lemma}

\begin{proof}
Define the coordinates $z^l_m(t)$ of $z(t)\in\R^d$ as follows:
$$z^l_m(t)=\left\{
\begin{array}{ll}
(t-\l_{i+j})&\mbox{ if } m> j, l+m=i+j, \ \l_{i+j}\le t\\
(t-\l_{i+j-1})&\mbox{ if } m> j, l+m=i+j-1, \ \l_{i+j-1}\le t\\
\vdots&\vdots\\
(t-\l_{j+2})&\mbox{ if } m> j, l+m=j+2, \ \l_{j+2}\le t\\
0&\mbox{ otherwise }\\
\end{array}
\right.
.$$
In particular, $z(t)=0$ if $i=1$.
The statement of the lemma now follows by direct calculation from the definition of $GZ(\l)$
and $G_k(\l)$.
\end{proof}

Lemmas \ref{l.CG} and \ref{l.CG_GZ} together with the backward induction on $k$ immediately yield an
elementary proof of the following theorem.
\begin{thm}\label{t.comparison}
Polytopes $F_k(\l)$ and $G_k(\l)$ have the same Ehrhart polynomial for all $k=0$,\ldots, $d$.
In particular, Gelfand--Zetlin polytope $GZ(\lambda)$ and Feigin--Fourier--Littelmann--Vinberg
polytope $FFLV(\l)$ have the same Ehrhart polynomial.
\end{thm}
The last statement of the theorem also follows from \cite{FFL}.
The first elementary proof of this statement was given in \cite{ABS} using a different approach.

Lemmas \ref{l.CG} and \ref{l.CG_GZ} imply that both $FFLV(\l)$ and $GZ(\l)$ can be obtained
inductively from a point by iterating the construction of Section \ref{ss.fiber}.
Note that both $F_{k-1}(\l)$ and $G_{k-1}(\l)$ fiber over a segment of length $\l_j-\l_{i+j}$,
and fibers are equal (up to a parallel translation) to $F_k(\mu(t))$ and $G_k(\mu(t))$,
respectively, for the same piecewise linear function $\mu(t)$ on the segment.
The only difference between these two cases is the presence of the shift vector $z(t)$ in the
second case.

\begin{example}{cf. \cite{Fo}} \label{e.FvsG}
For $n=3$, $k=0$, \ldots, $3$, and $n=4$, $k=2$, \ldots, $6$,
there exists a unimodular change of coordinates that maps $F_k$ to $G_k$.
Let $n=4$, and $k=1$.
Then $F_k$ provides the minimal example when $F_k$ is not combinatorially equivalent to
$G_k$.

We now illustrate how to obtain the inequalities defining  $F_1$ from those of $F_2$
using Lemma \ref{l.CG} or equivalently the construction of Section \ref{ss.fiber}
(and not the definition of $F_1$).
For $k=2$, we have $i=j=2$, and
$$
\mu(t)=\left\{
\begin{array}{ll}
(\l_1,\l_2,\l_3,t)&
\mbox{ if } \l_4\le t\le \l_3\\
(\l_1,\l_2,\: \ t,t)&
\mbox{ if } \l_3\le t\le \l_2\\
\end{array}
\right.
.
$$
By Example \ref{e.FFLV} the inequalities defining $F_2$ are
$$0\le u^1_1\le\l_1-\l_2; \quad 0\le u^1_2\le\l_2-\l_3; \quad
0\le u^2_1, \ u^3_1;$$
$$u^1_1+u^2_1+u^1_2\le\l_1-\l_3; \quad u^1_1+u^2_1+u^3_1\le\l_1-\l_4.$$
Put $u^2_2:=t-\l_4$.
Using the last statement of Lemma \ref{l.CG} as a definition of $F_1$, we get that $F_1$
is defined by  inequalities:
$$0\le u^1_1\le\l_1-\l_2; \quad 0\le u^1_2\le\l_2-\mu_3(u^2_2+\l_4); \quad
0\le u^2_1, \ u^3_1;$$
$$u^1_1+u^2_1+u^1_2\le\l_1-\mu_3(u^2_2+\l_4); \quad u^1_1+u^2_1+u^3_1\le\l_1-(u^2_2+\l_4);$$
$$0\le u^2_2\le\l_2-\l_4.$$
Using that $\mu_3(t)=\max\{\l_3,t\}$ and eliminating redundant inequalities we get
$$0\le u^1_1\le\l_1-\l_2; \quad 0\le u^1_2\le\l_2-\l_3; \quad u^1_2+u^2_2\le \l_2-\l_4; \quad
0\le u^2_1, \ u^3_1, \ u^2_2;$$
$$u^1_1+u^2_1+u^1_2\le\l_1-\l_3; \quad u^1_1+u^2_1+u^1_2+u^2_2\le\l_1-\l_4;$$
$$u^1_1+u^2_1+u^3_1+u^2_2\le\l_1-\l_4.$$
Similarly, one can restore $G_1$ from $G_2$ and check that there are only 10 inequalities for $G_1$.
\end{example}

\end{document}